\documentclass{amsart}
\usepackage{amsmath,amssymb,graphicx,latexsym}
\usepackage[usenames,dvipsnames]{color}
\usepackage{gensymb}		
\usepackage{tikz}
\usetikzlibrary{shapes}
\usepackage{hyperref}

\newtheorem{theorem}{Theorem}
\newtheorem{corollary}[theorem]{Corollary}
\newtheorem{proposition}[theorem]{Proposition}

\theoremstyle{definition}
\newtheorem*{notation*}{Notation}

\newcommand{\nequiv}{\mathrel{\not\equiv}}

\newcommand{\gray}[1]{\textcolor{Gray}{#1}}
\newcommand{\grayparen}[1]{\gray{(}{#1}\gray{)}}
\newcommand{\graysizedparen}[1]{\gray{\left( \textcolor{Black}{#1} \right)}}

\begin{document}

\title{Lucas' theorem modulo $p^2$}
\author{Eric Rowland}
\address{
	Department of Mathematics \\
	Hofstra University \\
	Hempstead, NY \\
	USA
}
\date{February 11, 2022}

\begin{abstract}
Lucas' theorem describes how to reduce a binomial coefficient $\binom{a}{b}$ modulo~$p$ by breaking off the least significant digits of $a$ and $b$ in base $p$.
We characterize the pairs of these digits for which Lucas' theorem holds modulo~$p^2$.
This characterization is naturally expressed using symmetries of Pascal's triangle.
\end{abstract}

\maketitle

\section{Introduction}

In 1878, Lucas~\cite{Lucas} discovered a formula for computing the residue of a binomial coefficient modulo~$p$, where $p$ is a prime.
Namely, if $r, s \in \{0, 1, \dots, p - 1\}$ and $a$ and $b$ are nonnegative integers, then
\begin{equation}\label{Lucas congruence}
	\binom{p a + r}{p b + s} \equiv \binom{a}{b} \binom{r}{s} \mod p.
\end{equation}
This congruence can also be written using base-$p$ representations.
Let the base-$p$ representations of $a$ and $b$ be $a_\ell \cdots a_1 a_0$ and $b_\ell \cdots b_1 b_0$, where we have made them the same length by padding the shorter representation with $0$s if necessary.
Iterating Congruence~\eqref{Lucas congruence} gives
\[
	\binom{a}{b} \equiv \binom{a_\ell}{b_\ell} \cdots \binom{a_1}{b_1} \binom{a_0}{b_0} \mod p.
\]

Several variants and generalizations of Lucas' theorem are known.
Me\v{s}trovi\'c~\cite{Mestrovic} gives an excellent survey.
In particular, it is natural to ask for Lucas-type congruences modulo higher powers of $p$.
We refer to a congruence of the form
\begin{equation}\label{general Lucas congruence}
	\binom{p a + r}{p b + s} \equiv \binom{a}{b} \binom{r}{s} \mod p^\alpha
\end{equation}
where $r, s \in \{0, 1, \dots, p - 1\}$ as a \emph{Lucas congruence}.
This congruence does not hold in general, but it does hold for certain values of $\alpha, p, r, s, a, b$.
Even prior to Lucas' work, Babbage~\cite{Babbage} in 1819 showed that
\begin{equation}\label{Babbage congruence}
	\binom{2 p - 1}{p - 1} \equiv 1 \mod p^2
\end{equation}
for all $p \geq 3$; this is a Lucas congruence where $r = s = p - 1$, $a = 1$, and $b = 0$.
In 1862, Wolstenholme~\cite{Wolstenholme} showed that Babbage's congruence holds modulo~$p^3$ if $p \geq 5$.
This was generalized by Glaisher~\cite[page 21]{Glaisher} in 1900 to the Lucas congruence
\begin{equation}
	\binom{p a - 1}{p - 1} \equiv 1 \mod p^3
\end{equation}
for all $a \geq 1$, again for $p \geq 5$.
Since $a \binom{p a - 1}{p - 1} = \binom{p a}{p}$, this implies $\binom{p a}{p} \equiv a \mod p^3$, which itself can be generalized to the Lucas congruence
\begin{equation}\label{Jacobsthal congruence}
	\binom{p a}{p b} \equiv \binom{a}{b} \mod p^3
\end{equation}
for $a \geq 0$, $b \geq 0$, and $p \geq 5$.
Congruence~\eqref{Jacobsthal congruence} is often attributed to Ljunggren~\cite{Ljunggren--Jacobsthal}.
However, Ljunggren only considered the special case $a = p b$ and was primarily interested in the case $a = p^n, b = p^{n - 1}$.
The general form seems to have been first obtained by Jacobsthal in the same paper~\cite{Ljunggren--Jacobsthal}.
It was independently rediscovered several times, including by Kazandzidis~\cite{Kazandzidis} and Bailey~\cite{Bailey}.
Siong~\cite{Siong} also gave a proof that Congruence~\eqref{Jacobsthal congruence} follows from Glaisher's congruence.
For $p = 2$ and $p = 3$, Congruence~\eqref{Jacobsthal congruence} does not hold modulo~$p^3$ in general but does hold modulo~$p^2$.

While each of the congruences \eqref{Babbage congruence}--\eqref{Jacobsthal congruence} uses a single pair $(r, s)$ of digits, some results of Bailey~\cite{Bailey} allow these digits to be general.
For every prime $p$, Bailey proved that
\[
	\binom{p^2 a + r}{p^2 b + s} \equiv \binom{a}{b} \binom{r}{s} \mod p^2
\]
for all $r, s \in \{0, 1, \dots, p - 1\}$, $a \geq 0$, and $b \geq 0$.
The equivalent form $\binom{p \grayparen{p a} + r}{p \grayparen{p b} + s} \equiv \binom{p a}{p b} \binom{r}{s} \mod p^2$ is a Lucas congruence.
For $p \geq 5$, Bailey also proved
\[
	\binom{p^3 a + r}{p^3 b + s} \equiv \binom{a}{b} \binom{r}{s} \mod p^3.
\]
These exponents $3$ were subsequently increased by Davis and Webb~\cite{Davis--Webb 1993}.
A further extension was found by Zhao~\cite{Zhao}, and generalizations of Lucas' theorem modulo~$p^\alpha$ for general $\alpha \geq 1$ were given by Davis and Webb~\cite{Davis--Webb 1990}, Granville~\cite{Granville}, and Yassawi and the author~\cite[Theorem~5.3]{Rowland--Yassawi}, although these results depart from the form of Congruence~\eqref{general Lucas congruence}.

In this article we consider the following question.
For which pairs $(r, s)$ of base-$p$ digits does the Lucas congruence
\[
	\binom{p a + r}{p b + s} \equiv \binom{a}{b} \binom{r}{s} \mod p^2
\]
hold for all $a \geq 0$ and $b \geq 0$?
The set of such pairs is our primary object of interest.

\begin{notation*}
For each prime $p$, let
\begin{multline*}
	D(p) =
	\bigg\{(r, s) \in \{0, 1, \dots, p - 1\}^2 : \\
	\text{$\binom{p a + r}{p b + s} \equiv \binom{a}{b} \binom{r}{s} \mod p^2$ for all $a \geq 0, b \geq 0$}\bigg\}.
\end{multline*}
\end{notation*}

\section{Description of the set $D(p)$}

Congruence~\eqref{Jacobsthal congruence} implies that $D(p)$ is nonempty for each prime $p \geq 5$, since $(0, 0) \in D(p)$.
Computer experiments suggest that $D(p)$ contains additional pairs as well.
For example, we will show that $D(3) = \{(0, 0), (2, 0), (2, 2)\}$ and $D(7) = \{(0, 0), (4, 2), (6, 0), (6, 6)\}$.
The following table highlights the binomial coefficients $\binom{r}{s}$ corresponding to points $(r, s) \in D(7)$.
\begin{center}
\large
\resizebox{.4\textwidth}{!}{
\begin{tikzpicture}[
	square/.style={shape=regular polygon, regular polygon sides=4, minimum size=1.0607cm, draw=black!60, inner sep=0, anchor=south, very thin},
	filledsquare/.style={shape=regular polygon, regular polygon sides=4, minimum size=1.0607cm, draw=lightgray!70!blue!80, inner sep=0, anchor=south, fill=lightgray!70!blue, fill opacity=.5, very thin}
]
\pgfmathsetmacro\rowcountminusone{6}
\foreach \j in {0,...,\rowcountminusone}{
	\foreach \i in {0,...,\rowcountminusone}{
		\node[square] (s\i;\j) at ({\i*3/4},{\j*3/4 - 4.885}) {};
	}
}
\node[filledsquare] (s0;0) at ({0*3/4},{0*3/4 - 4.885}) {};
\node[filledsquare] (s0;6) at ({0*3/4},{6*3/4 - 4.885}) {};
\node[filledsquare] (s2;2) at ({2*3/4},{2*3/4 - 4.885}) {};
\node[filledsquare] (s6;0) at ({6*3/4},{0*3/4 - 4.885}) {};
\node[] at (0*3/2,0*3/2) {$1$};
\node[] at (.5*3/2,0*3/2) {\gray{$0$}};
\node[] at (1*3/2,0*3/2) {\gray{$0$}};
\node[] at (1.5*3/2,0*3/2) {\gray{$0$}};
\node[] at (2*3/2,0*3/2) {\gray{$0$}};
\node[] at (2.5*3/2,0*3/2) {\gray{$0$}};
\node[] at (3*3/2,0*3/2) {\gray{$0$}};
\node[] at (0*3/2,-.5*3/2) {$1$};
\node[] at (.5*3/2,-.5*3/2) {$1$};
\node[] at (1*3/2,-.5*3/2) {\gray{$0$}};
\node[] at (1.5*3/2,-.5*3/2) {\gray{$0$}};
\node[] at (2*3/2,-.5*3/2) {\gray{$0$}};
\node[] at (2.5*3/2,-.5*3/2) {\gray{$0$}};
\node[] at (3*3/2,-.5*3/2) {\gray{$0$}};
\node[] at (0*3/2,-1*3/2) {$1$};
\node[] at (.5*3/2,-1*3/2) {$2$};
\node[] at (1*3/2,-1*3/2) {$1$};
\node[] at (1.5*3/2,-1*3/2) {\gray{$0$}};
\node[] at (2*3/2,-1*3/2) {\gray{$0$}};
\node[] at (2.5*3/2,-1*3/2) {\gray{$0$}};
\node[] at (3*3/2,-1*3/2) {\gray{$0$}};
\node[] at (0*3/2,-1.5*3/2) {$1$};
\node[] at (.5*3/2,-1.5*3/2) {$3$};
\node[] at (1*3/2,-1.5*3/2) {$3$};
\node[] at (1.5*3/2,-1.5*3/2) {$1$};
\node[] at (2*3/2,-1.5*3/2) {\gray{$0$}};
\node[] at (2.5*3/2,-1.5*3/2) {\gray{$0$}};
\node[] at (3*3/2,-1.5*3/2) {\gray{$0$}};
\node[] at (0*3/2,-2*3/2) {$1$};
\node[] at (.5*3/2,-2*3/2) {$4$};
\node[] at (1*3/2,-2*3/2) {$6$};
\node[] at (1.5*3/2,-2*3/2) {$4$};
\node[] at (2*3/2,-2*3/2) {$1$};
\node[] at (2.5*3/2,-2*3/2) {\gray{$0$}};
\node[] at (3*3/2,-2*3/2) {\gray{$0$}};
\node[] at (0*3/2,-2.5*3/2) {$1$};
\node[] at (.5*3/2,-2.5*3/2) {$5$};
\node[] at (1*3/2,-2.5*3/2) {$10$};
\node[] at (1.5*3/2,-2.5*3/2) {$10$};
\node[] at (2*3/2,-2.5*3/2) {$5$};
\node[] at (2.5*3/2,-2.5*3/2) {$1$};
\node[] at (3*3/2,-2.5*3/2) {\gray{$0$}};
\node[] at (0*3/2,-3*3/2) {$1$};
\node[] at (.5*3/2,-3*3/2) {$6$};
\node[] at (1*3/2,-3*3/2) {$15$};
\node[] at (1.5*3/2,-3*3/2) {$20$};
\node[] at (2*3/2,-3*3/2) {$15$};
\node[] at (2.5*3/2,-3*3/2) {$6$};
\node[] at (3*3/2,-3*3/2) {$1$};
\end{tikzpicture}
}
\normalsize
\end{center}

Our first result is that the zeros in this table do not correspond to points in $D(p)$.

\begin{proposition}\label{zero region}
Let $p$ be a prime.
If $s > r$, then $(r, s) \notin D(p)$.
\end{proposition}

\begin{proof}
Let $a = 1$ and $b = 0$.
The binomial coefficient $\binom{p a + r}{p b + s} = \binom{p + r}{s} = \frac{\grayparen{p + r}!}{s! \grayparen{p + r - s}!}$ is divisible by $p$ but not $p^2$.
On the other hand, $\binom{a}{b} \binom{r}{s} = \binom{r}{s} = 0$ is divisible by $p^2$.
Therefore $\binom{p a + r}{p b + s} \nequiv \binom{a}{b} \binom{r}{s} \mod p^2$.
\end{proof}

In light of Proposition~\ref{zero region}, we omit points $(r, s)$ where $s > r$ from the previous table.
Then we shear the remaining triangle:
\begin{center}
\large
\resizebox{.4\textwidth}{!}{
\begin{tikzpicture}[
	hexagon/.style={shape=regular polygon, regular polygon sides=6, minimum size=1cm, draw, inner sep=0, anchor=south, rotate=30, ultra thin},
	filledhexagon/.style={shape=regular polygon, regular polygon sides=6, minimum size=1cm, draw=lightgray!70!blue!80, inner sep=0, anchor=south, fill=lightgray!70!blue, fill opacity=.5, rotate=30, ultra thin}
]
\pgfmathsetmacro\rowcountminusone{6}
\foreach \j in {0,...,\rowcountminusone}{
	\pgfmathsetmacro\rowcountminusoneminusj{\rowcountminusone-\j}
	\foreach \i in {0,...,\rowcountminusoneminusj}{
		\node[hexagon] (h\i;\j) at ({(\i+\j/2)*sin(60) + .21},{\j*3/4 - .37}) {};
	}
}
\node[filledhexagon] (f0;6) at ({(0+6/2)*sin(60) + .21},{6*3/4 - .37}) {};
\node[filledhexagon] (f2;2) at ({(2+2/2)*sin(60) + .21},{2*3/4 - .37}) {};
\node[filledhexagon] (f0;0) at ({(0+0/2)*sin(60) + .21},{0*3/4 - .37}) {};
\node[filledhexagon] (f6;0) at ({(6+0/2)*sin(60) + .21},{0*3/4 - .37}) {};
\node[] at ({(0+6/2)*sin(60)},{6*3/4}) {$1$};
\node[] at ({(0+5/2)*sin(60)},{5*3/4}) {$1$};
\node[] at ({(1+5/2)*sin(60)},{5*3/4}) {$1$};
\node[] at ({(0+4/2)*sin(60)},{4*3/4}) {$1$};
\node[] at ({(1+4/2)*sin(60)},{4*3/4}) {$2$};
\node[] at ({(2+4/2)*sin(60)},{4*3/4}) {$1$};
\node[] at ({(0+3/2)*sin(60)},{3*3/4}) {$1$};
\node[] at ({(1+3/2)*sin(60)},{3*3/4}) {$3$};
\node[] at ({(2+3/2)*sin(60)},{3*3/4}) {$3$};
\node[] at ({(3+3/2)*sin(60)},{3*3/4}) {$1$};
\node[] at ({(0+2/2)*sin(60)},{2*3/4}) {$1$};
\node[] at ({(1+2/2)*sin(60)},{2*3/4}) {$4$};
\node[] at ({(2+2/2)*sin(60)},{2*3/4}) {$6$};
\node[] at ({(3+2/2)*sin(60)},{2*3/4}) {$4$};
\node[] at ({(4+2/2)*sin(60)},{2*3/4}) {$1$};
\node[] at ({(0+1/2)*sin(60)},{1*3/4}) {$1$};
\node[] at ({(1+1/2)*sin(60)},{1*3/4}) {$5$};
\node[] at ({(2+1/2)*sin(60)},{1*3/4}) {$10$};
\node[] at ({(3+1/2)*sin(60)},{1*3/4}) {$10$};
\node[] at ({(4+1/2)*sin(60)},{1*3/4}) {$5$};
\node[] at ({(5+1/2)*sin(60)},{1*3/4}) {$1$};
\node[] at ({(0+0/2)*sin(60)},{0*3/4}) {$1$};
\node[] at ({(1+0/2)*sin(60)},{0*3/4}) {$6$};
\node[] at ({(2+0/2)*sin(60)},{0*3/4}) {$15$};
\node[] at ({(3+0/2)*sin(60)},{0*3/4}) {$20$};
\node[] at ({(4+0/2)*sin(60)},{0*3/4}) {$15$};
\node[] at ({(5+0/2)*sin(60)},{0*3/4}) {$6$};
\node[] at ({(6+0/2)*sin(60)},{0*3/4}) {$1$};
\end{tikzpicture}
}
\normalsize
\end{center}
For $p = 11$, the set $D(11)$ contains $9$ pairs of digits, arranged as follows.
\begin{center}
\large
\resizebox{.4\textwidth}{!}{
\begin{tikzpicture}[
	hexagon/.style={shape=regular polygon, regular polygon sides=6, minimum size=1cm, draw, inner sep=0, anchor=south, rotate=30, ultra thin},
	filledhexagon/.style={shape=regular polygon, regular polygon sides=6, minimum size=1cm, draw=lightgray!70!blue!80, inner sep=0, anchor=south, fill=lightgray!70!blue, fill opacity=.5, rotate=30, ultra thin}
]
\pgfmathsetmacro\rowcountminusone{10}
\foreach \j in {0,...,\rowcountminusone}{
	\pgfmathsetmacro\rowcountminusoneminusj{\rowcountminusone-\j}
	\foreach \i in {0,...,\rowcountminusoneminusj}{
		\node[hexagon] (h\i;\j) at ({(\i+\j/2)*sin(60) + .21},{\j*3/4 - .37}) {};
	}
}
\node[filledhexagon] (f0;10) at ({(0+10/2)*sin(60) + .21},{10*3/4 - .37}) {};
\node[filledhexagon] (f0;7) at ({(0+7/2)*sin(60) + .21},{7*3/4 - .37}) {};
\node[filledhexagon] (f3;7) at ({(3+7/2)*sin(60) + .21},{7*3/4 - .37}) {};
\node[filledhexagon] (f0;3) at ({(0+3/2)*sin(60) + .21},{3*3/4 - .37}) {};
\node[filledhexagon] (f7;3) at ({(7+3/2)*sin(60) + .21},{3*3/4 - .37}) {};
\node[filledhexagon] (f0;0) at ({(0+0/2)*sin(60) + .21},{0*3/4 - .37}) {};
\node[filledhexagon] (f3;0) at ({(3+0/2)*sin(60) + .21},{0*3/4 - .37}) {};
\node[filledhexagon] (f7;0) at ({(7+0/2)*sin(60) + .21},{0*3/4 - .37}) {};
\node[filledhexagon] (f10;0) at ({(10+0/2)*sin(60) + .21},{0*3/4 - .37}) {};
\node[] at ({(0+10/2)*sin(60)},{10*3/4}) {$1$};
\node[] at ({(0+9/2)*sin(60)},{9*3/4}) {$1$};
\node[] at ({(1+9/2)*sin(60)},{9*3/4}) {$1$};
\node[] at ({(0+8/2)*sin(60)},{8*3/4}) {$1$};
\node[] at ({(1+8/2)*sin(60)},{8*3/4}) {$2$};
\node[] at ({(2+8/2)*sin(60)},{8*3/4}) {$1$};
\node[] at ({(0+7/2)*sin(60)},{7*3/4}) {$1$};
\node[] at ({(1+7/2)*sin(60)},{7*3/4}) {$3$};
\node[] at ({(2+7/2)*sin(60)},{7*3/4}) {$3$};
\node[] at ({(3+7/2)*sin(60)},{7*3/4}) {$1$};
\node[] at ({(0+6/2)*sin(60)},{6*3/4}) {$1$};
\node[] at ({(1+6/2)*sin(60)},{6*3/4}) {$4$};
\node[] at ({(2+6/2)*sin(60)},{6*3/4}) {$6$};
\node[] at ({(3+6/2)*sin(60)},{6*3/4}) {$4$};
\node[] at ({(4+6/2)*sin(60)},{6*3/4}) {$1$};
\node[] at ({(0+5/2)*sin(60)},{5*3/4}) {$1$};
\node[] at ({(1+5/2)*sin(60)},{5*3/4}) {$5$};
\node[] at ({(2+5/2)*sin(60)},{5*3/4}) {$10$};
\node[] at ({(3+5/2)*sin(60)},{5*3/4}) {$10$};
\node[] at ({(4+5/2)*sin(60)},{5*3/4}) {$5$};
\node[] at ({(5+5/2)*sin(60)},{5*3/4}) {$1$};
\node[] at ({(0+4/2)*sin(60)},{4*3/4}) {$1$};
\node[] at ({(1+4/2)*sin(60)},{4*3/4}) {$6$};
\node[] at ({(2+4/2)*sin(60)},{4*3/4}) {$15$};
\node[] at ({(3+4/2)*sin(60)},{4*3/4}) {$20$};
\node[] at ({(4+4/2)*sin(60)},{4*3/4}) {$15$};
\node[] at ({(5+4/2)*sin(60)},{4*3/4}) {$6$};
\node[] at ({(6+4/2)*sin(60)},{4*3/4}) {$1$};
\node[] at ({(0+3/2)*sin(60)},{3*3/4}) {$1$};
\node[] at ({(1+3/2)*sin(60)},{3*3/4}) {$7$};
\node[] at ({(2+3/2)*sin(60)},{3*3/4}) {$21$};
\node[] at ({(3+3/2)*sin(60)},{3*3/4}) {$35$};
\node[] at ({(4+3/2)*sin(60)},{3*3/4}) {$35$};
\node[] at ({(5+3/2)*sin(60)},{3*3/4}) {$21$};
\node[] at ({(6+3/2)*sin(60)},{3*3/4}) {$7$};
\node[] at ({(7+3/2)*sin(60)},{3*3/4}) {$1$};
\node[] at ({(0+2/2)*sin(60)},{2*3/4}) {$1$};
\node[] at ({(1+2/2)*sin(60)},{2*3/4}) {$8$};
\node[] at ({(2+2/2)*sin(60)},{2*3/4}) {$28$};
\node[] at ({(3+2/2)*sin(60)},{2*3/4}) {$56$};
\node[] at ({(4+2/2)*sin(60)},{2*3/4}) {$70$};
\node[] at ({(5+2/2)*sin(60)},{2*3/4}) {$56$};
\node[] at ({(6+2/2)*sin(60)},{2*3/4}) {$28$};
\node[] at ({(7+2/2)*sin(60)},{2*3/4}) {$8$};
\node[] at ({(8+2/2)*sin(60)},{2*3/4}) {$1$};
\node[] at ({(0+1/2)*sin(60)},{1*3/4}) {$1$};
\node[] at ({(1+1/2)*sin(60)},{1*3/4}) {$9$};
\node[] at ({(2+1/2)*sin(60)},{1*3/4}) {$36$};
\node[] at ({(3+1/2)*sin(60)},{1*3/4}) {$84$};
\node[] at ({(4+1/2)*sin(60)},{1*3/4}) {$126$};
\node[] at ({(5+1/2)*sin(60)},{1*3/4}) {$126$};
\node[] at ({(6+1/2)*sin(60)},{1*3/4}) {$84$};
\node[] at ({(7+1/2)*sin(60)},{1*3/4}) {$36$};
\node[] at ({(8+1/2)*sin(60)},{1*3/4}) {$9$};
\node[] at ({(9+1/2)*sin(60)},{1*3/4}) {$1$};
\node[] at ({(0+0/2)*sin(60)},{0*3/4}) {$1$};
\node[] at ({(1+0/2)*sin(60)},{0*3/4}) {$10$};
\node[] at ({(2+0/2)*sin(60)},{0*3/4}) {$45$};
\node[] at ({(3+0/2)*sin(60)},{0*3/4}) {$120$};
\node[] at ({(4+0/2)*sin(60)},{0*3/4}) {$210$};
\node[] at ({(5+0/2)*sin(60)},{0*3/4}) {$252$};
\node[] at ({(6+0/2)*sin(60)},{0*3/4}) {$210$};
\node[] at ({(7+0/2)*sin(60)},{0*3/4}) {$120$};
\node[] at ({(8+0/2)*sin(60)},{0*3/4}) {$45$};
\node[] at ({(9+0/2)*sin(60)},{0*3/4}) {$10$};
\node[] at ({(10+0/2)*sin(60)},{0*3/4}) {$1$};
\end{tikzpicture}
}
\normalsize
\end{center}
For $p = 17$, $p = 29$, and $p = 37$ the pairs in $D(p)$ appear in the following locations.
\begin{center}
\resizebox{.25\textwidth}{!}{
\begin{tikzpicture}[
	hexagon/.style={shape=regular polygon, regular polygon sides=6, minimum size=1cm, draw=black!30, inner sep=0, anchor=south, rotate=30, ultra thin},
	filledhexagon/.style={shape=regular polygon, regular polygon sides=6, minimum size=1cm, draw=lightgray!70!blue!80, inner sep=0, anchor=south, fill=lightgray!70!blue, fill opacity=.8, rotate=30, ultra thin}
]
\pgfmathsetmacro\rowcountminusone{16}
\foreach \j in {0,...,\rowcountminusone}{
	\pgfmathsetmacro\rowcountminusoneminusj{\rowcountminusone-\j}
	\foreach \i in {0,...,\rowcountminusoneminusj}{
		\node[hexagon] (h\i;\j) at ({(\i+\j/2)*sin(60)},{\j*3/4}) {};
	}
}
\node[filledhexagon] (f0;0) at ({(0+0/2)*sin(60)},{0*3/4}) {};
\node[filledhexagon] (f7;2) at ({(7+2/2)*sin(60)},{2*3/4}) {};
\node[filledhexagon] (f2;7) at ({(2+7/2)*sin(60)},{7*3/4}) {};
\node[filledhexagon] (f7;7) at ({(7+7/2)*sin(60)},{7*3/4}) {};
\node[filledhexagon] (f16;0) at ({(16+0/2)*sin(60)},{0*3/4}) {};
\node[filledhexagon] (f0;16) at ({(0+16/2)*sin(60)},{16*3/4}) {};
\end{tikzpicture}
}
\qquad \qquad
\resizebox{.25\textwidth}{!}{
\begin{tikzpicture}[
	hexagon/.style={shape=regular polygon, regular polygon sides=6, minimum size=1cm, draw=black!30, inner sep=0, anchor=south, rotate=30, ultra thin},
	filledhexagon/.style={shape=regular polygon, regular polygon sides=6, minimum size=1cm, draw=lightgray!70!blue!80, inner sep=0, anchor=south, fill=lightgray!70!blue, fill opacity=.8, rotate=30, ultra thin}
]
\pgfmathsetmacro\rowcountminusone{28}
\foreach \j in {0,...,\rowcountminusone}{
	\pgfmathsetmacro\rowcountminusoneminusj{\rowcountminusone-\j}
	\foreach \i in {0,...,\rowcountminusoneminusj}{
		\node[hexagon] (h\i;\j) at ({(\i+\j/2)*sin(60)},{\j*3/4}) {};
	}
}
\node[filledhexagon] (f0;0) at ({(0+0/2)*sin(60)},{0*3/4}) {};
\node[filledhexagon] (f13;0) at ({(13+0/2)*sin(60)},{0*3/4}) {};
\node[filledhexagon] (f0;13) at ({(0+13/2)*sin(60)},{13*3/4}) {};
\node[filledhexagon] (f15;0) at ({(15+0/2)*sin(60)},{0*3/4}) {};
\node[filledhexagon] (f0;15) at ({(0+15/2)*sin(60)},{15*3/4}) {};
\node[filledhexagon] (f28;0) at ({(28+0/2)*sin(60)},{0*3/4}) {};
\node[filledhexagon] (f15;13) at ({(15+13/2)*sin(60)},{13*3/4}) {};
\node[filledhexagon] (f13;15) at ({(13+15/2)*sin(60)},{15*3/4}) {};
\node[filledhexagon] (f0;28) at ({(0+28/2)*sin(60)},{28*3/4}) {};
\end{tikzpicture}
}
\qquad \qquad
\resizebox{.25\textwidth}{!}{
\begin{tikzpicture}[
	hexagon/.style={shape=regular polygon, regular polygon sides=6, minimum size=1cm, draw=black!30, inner sep=0, anchor=south, rotate=30, ultra thin},
	filledhexagon/.style={shape=regular polygon, regular polygon sides=6, minimum size=1cm, draw=lightgray!70!blue!80, inner sep=0, anchor=south, fill=lightgray!70!blue, fill opacity=.8, rotate=30, ultra thin}
]
\pgfmathsetmacro\rowcountminusone{36}
\foreach \j in {0,...,\rowcountminusone}{
	\pgfmathsetmacro\rowcountminusoneminusj{\rowcountminusone-\j}
	\foreach \i in {0,...,\rowcountminusoneminusj}{
		\node[hexagon] (h\i;\j) at ({(\i+\j/2)*sin(60)},{\j*3/4}) {};
	}
}
\node[filledhexagon] (f0;0) at ({(0+0/2)*sin(60)},{0*3/4}) {};
\node[filledhexagon] (f3;3) at ({(3+3/2)*sin(60)},{3*3/4}) {};
\node[filledhexagon] (f17;0) at ({(17+0/2)*sin(60)},{0*3/4}) {};
\node[filledhexagon] (f0;17) at ({(0+17/2)*sin(60)},{17*3/4}) {};
\node[filledhexagon] (f19;0) at ({(19+0/2)*sin(60)},{0*3/4}) {};
\node[filledhexagon] (f0;19) at ({(0+19/2)*sin(60)},{19*3/4}) {};
\node[filledhexagon] (f14;8) at ({(14+8/2)*sin(60)},{8*3/4}) {};
\node[filledhexagon] (f8;14) at ({(8+14/2)*sin(60)},{14*3/4}) {};
\node[filledhexagon] (f12;12) at ({(12+12/2)*sin(60)},{12*3/4}) {};
\node[filledhexagon] (f14;14) at ({(14+14/2)*sin(60)},{14*3/4}) {};
\node[filledhexagon] (f30;3) at ({(30+3/2)*sin(60)},{3*3/4}) {};
\node[filledhexagon] (f3;30) at ({(3+30/2)*sin(60)},{30*3/4}) {};
\node[filledhexagon] (f36;0) at ({(36+0/2)*sin(60)},{0*3/4}) {};
\node[filledhexagon] (f19;17) at ({(19+17/2)*sin(60)},{17*3/4}) {};
\node[filledhexagon] (f17;19) at ({(17+19/2)*sin(60)},{19*3/4}) {};
\node[filledhexagon] (f0;36) at ({(0+36/2)*sin(60)},{36*3/4}) {};
\end{tikzpicture}
} \\
$p = 17$ \hspace{3.44cm} $p = 29$ \hspace{3.44cm} $p = 37$
\end{center}
These pictures suggest that $D(p)$ is invariant under the symmetries of the equilateral triangle!

Reflection symmetry about the vertical axis is not altogether surprising, since Pascal's triangle also exhibits this symmetry.
We establish this in Proposition~\ref{reflection symmetry}.
However, the rotational symmetry of $D(p)$ is unexpected.

To identify the image of $(r, s)$ under rotation, we use the fact that counterclockwise rotation by $120\degree$ is equivalent to the composition of two reflections:
\begin{center}
	\includegraphics[width=.26\textwidth]{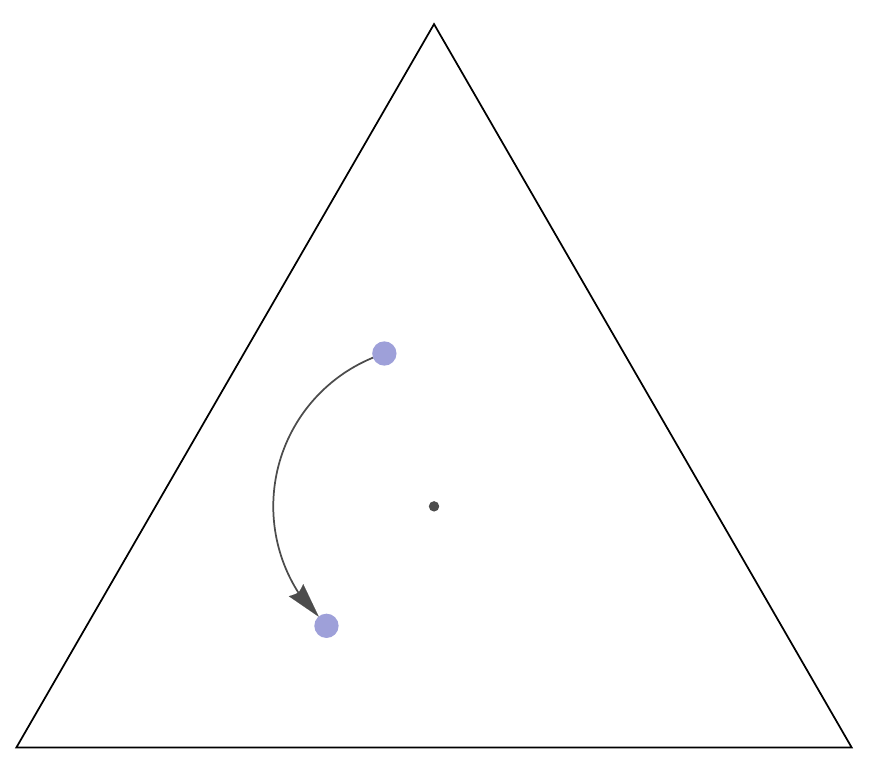}
	\hspace{1cm}
	\includegraphics[width=.26\textwidth]{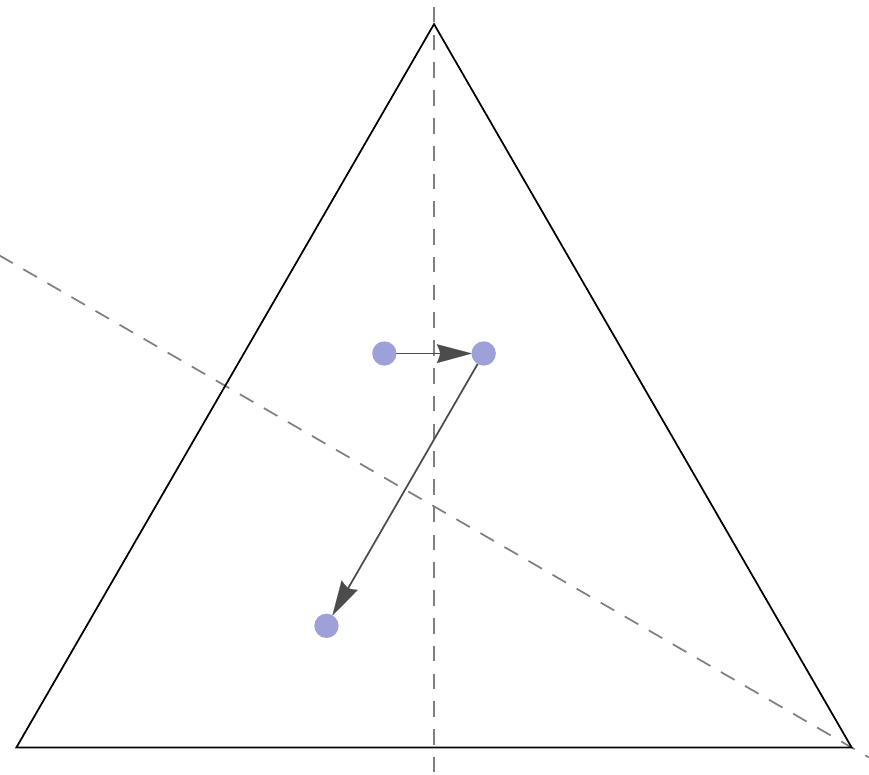}
\end{center}
The first reflection is through the vertical altitude of the triangle.
This reflection maps the point $(r, s)$ to $(r, r - s)$.
The second reflection is through the altitude passing through the lower right vertex.
This reflection maps $(r, s)$ to $(p - 1 - r + s, s)$, as can be seen by shearing so that this altitude is horizontal.
Composing these reflections shows that the rotation maps $(r, s)$ to $(p - 1 - s, r - s)$.
Therefore the three binomial coefficients visited by the orbit of $(r, s)$ under rotation by $120\degree$ are
\[
	\binom{r}{s}, \binom{p - 1 - s}{r - s}, \binom{p - 1 - r + s}{p - 1 - r},
\]
the third of which is equal to $\binom{p - 1 - r + s}{s}$.

In general, these three binomial coefficients are not equal, nor are they congruent modulo~$p$.
However, we will show in Corollary~\ref{rotation} that they do satisfy a congruence modulo~$p$ if we multiply them by the correct signs.
Furthermore, the elements of $D(p)$ can be characterized as the pairs $(r, s)$ for which this congruence holds not just modulo~$p$ but modulo~$p^2$.

\begin{theorem}\label{characterization}
Let $p$ be a prime, and let $r, s \in \{0, 1, \dots, p - 1\}$.
The congruence
\[
	\binom{p a + r}{p b + s} \equiv \binom{a}{b} \binom{r}{s} \mod p^2
\]
holds for all $a \geq 0$ and $b \geq 0$ if and only if $s \leq r$ and
\begin{equation}\label{rotation congruence}
	\binom{r}{s}
	\equiv \grayparen{-1}^{r - s} \binom{p - 1 - s}{r - s}
	\equiv \grayparen{-1}^s \binom{p - 1 - r + s}{s} \mod p^2.
\end{equation}
\end{theorem}

For certain classes of primes, $D(p)$ contains digit pairs that correspond to simple geometric points in the triangle.
For example, if $p \equiv 1 \mod 3$ then the center of the triangle has integer coordinates, namely $r = \frac{2}{3} \grayparen{p - 1}$ and $s = \frac{1}{3} \grayparen{p - 1}$.
Moreover, $p \equiv 1 \mod 6$ in this case, so the coordinates $r$ and $s$ are even, and $1 = \grayparen{-1}^{r - s} = \grayparen{-1}^s$.
Since the center is invariant under rotation about itself, the point $(r, s)$ satisfies Congruence~\eqref{rotation congruence}.
Consequently, $(r, s) \in D(p)$ and we obtain the following congruence.

\begin{corollary}\label{center point}
If $p \equiv 1 \mod 3$, then
\[
	\binom{p a + \frac{2}{3} \grayparen{p - 1}}{p b + \frac{1}{3} \grayparen{p - 1}}
	\equiv
	\binom{a}{b} \binom{\frac{2}{3} \grayparen{p - 1}}{\frac{1}{3} \grayparen{p - 1}} \mod p^2
\]
for all $a \geq 0$ and $b \geq 0$.
\end{corollary}

We can iterate Corollary~\ref{center point} for the particular numbers $a = \frac{2}{3} \grayparen{p - 1} \sum_{i = 0}^{\ell - 1} p^i = \frac{2}{3} \grayparen{p^\ell - 1}$ and $b = \frac{1}{3} \grayparen{p^\ell - 1}$ whose base-$p$ representations consist of $\ell$ copies of the digits $\frac{2}{3} \grayparen{p - 1}$ and $\frac{1}{3} \grayparen{p - 1}$, respectively.
Therefore, if $p \equiv 1 \mod 3$ and $\ell \geq 0$, then
\[
	\binom{\frac{2}{3} \grayparen{p^\ell - 1}}{\frac{1}{3} \grayparen{p^\ell - 1}}
	\equiv
	\binom{\frac{2}{3} \grayparen{p - 1}}{\frac{1}{3} \grayparen{p - 1}}^\ell \mod p^2.
\]
The value of $\binom{2 \grayparen{p - 1} / 3}{\grayparen{p - 1} / 3}$ modulo $p$ was studied by Jacobi, and its value modulo $p^2$ was shown by Yeung~\cite[Theorem~4.13]{Yeung} to be $\binom{2 \grayparen{p - 1} / 3}{\grayparen{p - 1} / 3} \equiv -A + \frac{p}{A} \mod p^2$, where $4 p = A^2 + 27 B^2$ and the sign of $A$ is chosen so that $A \equiv 1 \mod 3$.

A prime $p$ is a \emph{Wieferich prime} if $2^{p - 1} \equiv 1 \mod p^2$.
Only two such primes are known: $1093$ and $3511$.
It will follow from the proof of Theorem~\ref{characterization} that $p$ is a Wieferich prime if and only if $\{(\frac{p - 1}{2}, 0), (\frac{p - 1}{2}, \frac{p - 1}{2}), (p - 1, \frac{p - 1}{2})\} \subseteq D(p)$.
These digits pairs correspond to the midpoints of the three edges of the triangle.
Morley~\cite{Morley} proved that $\binom{p - 1}{\grayparen{p - 1}/2} \equiv (-1)^{\grayparen{p - 1}/2} 4^{p - 1} \mod p^3$ for every prime $p \geq 5$.
In particular, $\binom{p - 1}{\grayparen{p - 1}/2} \equiv (-1)^{\grayparen{p - 1}/2} \mod p^2$ for Wieferich primes.

An interesting question, which we do not address here, is this:
What else can be said about the size of $D(p)$ as a function of $p$?
The following table lists the elements of $D(p)$ for the first ten primes.
\[
	\begin{array}{r|l}
		p & D(p) \\
		\hline
		2 & \{(0, 0)\} \\
		3 & \{(0, 0), (2, 0), (2, 2)\} \\
		5 & \{(0, 0), (4, 0), (4, 4)\} \\
		7 & \{(0, 0), (4, 2), (6, 0), (6, 6)\} \\
		11 & \{(0, 0), (3, 0), (3, 3), (7, 0), (7, 7), (10, 0), (10, 3), (10, 7), (10, 10)\} \\
		13 & \{(0, 0), (8, 4), (12, 0), (12, 12)\} \\
		17 & \{(0, 0), (9, 2), (9, 7), (14, 7), (16, 0), (16, 16)\} \\
		19 & \{(0, 0), (12, 6), (18, 0), (18, 18)\} \\
		23 & \{(0, 0), (22, 0), (22, 22)\} \\
		29 & \{(0, 0), (13, 0), (13, 13), (15, 0), (15, 15), (28, 0), (28, 13), (28, 15), (28, 28)\}
	\end{array}
\]

Theorem~\ref{characterization} was suggested by an analogous result for the Ap\'ery numbers, which are defined by $A(n) = \sum_{k = 0}^n \binom{n}{k}^2 \binom{n + k}{k}^2$.
Gessel~\cite{Gessel} showed that the Ap\'ery numbers satisfy the one-dimensional Lucas congruence $A(p n + r) \equiv A(n) A(r) \mod p$ for all $r \in \{0, 1, \dots, p - 1\}$ and all $n \geq 0$.
For certain values of $r$, this congruence also holds modulo~$p^2$.
Gessel noticed that $A(3 n + r) \equiv A(n) A(r) \mod 9$ for all $r \in \{0, 1, 2\}$.
By computing an automaton for the Ap\'ery numbers modulo~$25$, Yassawi and the author~\cite[Theorem~3.31]{Rowland--Yassawi} showed that $A(5 n + r) \equiv A(n) A(r) \mod 25$ if $r \in \{0, 2, 4\}$.
This was recently generalized to all primes~\cite{Rowland--Yassawi--Krattenthaler}.
Namely, the digits $r \in \{0, 1, \dots, p - 1\}$ for which all $n \geq 0$ satisfy
\[
	A(p n + r) \equiv A(n) A(r) \mod p^2
\]
are precisely the digits for which $A(r) \equiv A(p - 1 - r) \mod p^2$.
The reflection symmetry $A(r) \equiv A(p - 1 - r) \mod p$ was established by Malik and Straub~\cite[Lemma~6.2]{Malik--Straub} for all $r \in \{0, 1, \dots, p - 1\}$.
Therefore, the elements of both $D(p)$ and the analogous set for the Ap\'ery numbers can be characterized as those for which a certain symmetry modulo~$p$ in fact holds modulo~$p^2$.

In light of Theorem~\ref{characterization}, it is natural to ask about digit pairs $(r, s)$ for which the Lucas congruence holds modulo $p^3$ for all $a \geq 0$ and $b \geq 0$.
Experiments suggest that for each prime $p \geq 5$ there are exactly three: $(0, 0)$, $(p - 1, 0)$, and $(p - 1, p - 1)$.
However, it is conceivable that certain primes support more.
We leave this as an open question.

\section{A general congruence}

To prove Theorem~\ref{characterization}, we first prove a general congruence for $\binom{p a + r}{p b + s}$ modulo~$p^2$.
Let $H_n = 1 + \frac{1}{2} + \frac{1}{3} + \dots + \frac{1}{n}$ be the $n$th harmonic number.
(In particular, the $0$th harmonic number is the empty sum $H_0 = 0$.)
For $r \in \{0, 1, \dots, p - 1\}$, the denominator of $H_r$ is not divisible by $p$, so we can interpret $H_n$ modulo~$p$ and modulo~$p^2$.

\begin{theorem}\label{general congruence}
Let $p$ be a prime.
If $0 \leq s \leq r \leq p - 1$, $a \geq 0$, and $b \geq 0$, then
\[
	\binom{p a + r}{p b + s}
	\equiv \binom{a}{b} \binom{r}{s} \graysizedparen{1 + p a \, \grayparen{H_r - H_{r - s}} + p b \, \grayparen{H_{r - s} - H_s}} \mod p^2.
\]
\end{theorem}

\begin{proof}
If $b > a$, then $\binom{p a + r}{p b + s} = 0 = \binom{a}{b}$, so the congruence holds.
Assume $b \leq a$.
By breaking a factorial into two products, we obtain
\begin{align*}
	\binom{p a + r}{p b + s}
	&= \frac{\grayparen{p a + r}!}{\grayparen{p b + s}! \grayparen{p a - p b + r - s}!} \\
	&= \frac{\grayparen{p a}!}{\grayparen{p b}! \grayparen{p a - p b}!} \frac{\prod_{i = 1}^r \grayparen{p a + i}}{\prod_{i = 1}^s \grayparen{p b + i} \prod_{i = 1}^{r - s} \grayparen{p a - p b + i}}.
\end{align*}
The first factor is $\binom{p a}{p b} \equiv \binom{a}{b} \mod p^2$; this is a special case of Congruence~\eqref{Jacobsthal congruence}.
In the second factor, we expand each product and collect terms by like powers of $p$.
Namely, $\prod_{i = 1}^r \grayparen{p a + i} \equiv r! + p a \sum_{i = 1}^r \frac{r!}{i} \mod p^2$.
This gives
\begin{align*}
	\binom{p a + r}{p b + s}
	&\equiv \binom{a}{b} \frac{r!}{s! \grayparen{r - s}!} \frac{1 + p a H_r}{\graysizedparen{1 + p b H_s} \graysizedparen{1 + p \, \grayparen{a - b} H_{r - s}}} \mod p^2 \\
	&\equiv \binom{a}{b} \binom{r}{s} \graysizedparen{1 + p a H_r} \graysizedparen{1 - p b H_s} \graysizedparen{1 - p \, \grayparen{a - b} H_{r - s}} \mod p^2 \\
	&\equiv \binom{a}{b} \binom{r}{s} \graysizedparen{1 + p a \, \grayparen{H_r - H_{r - s}} + p b \, \grayparen{H_{r - s} - H_s}} \mod p^2
\end{align*}
as desired.
\end{proof}

\section{Symmetries of $D(p)$}

In this section, we establish that $D(p)$ possesses the symmetries of the equilateral triangle.
In particular, we prove Theorem~\ref{characterization}.
The reflection symmetry $\binom{a}{b} = \binom{a}{a - b}$ of Pascal's triangle is familiar.
Next we show that $D(p)$ also exhibits this symmetry.

\begin{proposition}\label{reflection symmetry}
Let $p$ be a prime.
If $(r, s) \in D(p)$, then $(r, r - s) \in D(p)$.
\end{proposition}

\begin{proof}
Let $(r, s) \in D(p)$.
By Proposition~\ref{zero region}, $s \leq r$.
By assumption, $\binom{p a + r}{p b + s} \equiv \binom{a}{b} \binom{r}{s} \mod p^2$ for all $a \geq 0$ and $b \geq 0$.
Fix $a$ and $b$.
We would like to show $\binom{p a + r}{p b + r - s} \equiv \binom{a}{b} \binom{r}{r - s} \mod p^2$.
There are two cases.
If $b > a$, then $s < p \leq p \, \grayparen{b - a}$.
It follows that $p a + r < p b + r - s$.
Therefore $\binom{p a + r}{p b + r - s} = 0 = \binom{a}{b} \binom{r}{r - s}$, so the congruence holds.
On the other hand, if $b \leq a$, the reflection symmetry of Pascal's triangle gives
\[
	\binom{p a + r}{p b + r - s}
	= \binom{p a + r}{\grayparen{p a + r} - \grayparen{p b + r - s}}
	= \binom{p a + r}{p \, \grayparen{a - b} + s}.
\]
Since $(r, s) \in D(p)$, this implies
\begin{align*}
	\binom{p a + r}{p b + r - s}
	&\equiv \binom{a}{a - b} \binom{r}{s} \mod p^2 \\
	&= \binom{a}{b} \binom{r}{r - s} \\
	&\equiv \binom{p a}{p b} \binom{r}{r - s} \mod p^2,
\end{align*}
as desired.
In both cases, $\binom{p a + r}{p b + r - s} \equiv \binom{a}{b} \binom{r}{r - s} \mod p^2$, so $(r, r - s) \in D(p)$.
\end{proof}

In addition to the reflection symmetry, the first $p$ rows of Pascal's triangle also exhibit rotational symmetry modulo~$p$ up to sign.
To see this, first we prove the following congruence modulo~$p^2$.

\begin{proposition}\label{general rotation}
Let $p$ be a prime.
If $0 \leq s \leq r \leq p - 1$, then
\begin{equation}\label{rotation congruence modulo p^2}
	\binom{r}{s}
	\equiv \grayparen{-1}^{r - s} \binom{p - 1 - s}{r - s} \grayparen{1 + p H_r - p H_s} \mod p^2.
\end{equation}
\end{proposition}

\begin{proof}
Similar to the proof of Theorem~\ref{general congruence}, we expand the product $\grayparen{p - 1 - r}!$ and collect terms by like powers of $p$:
\begin{align*}
	r! \grayparen{p - 1 - r}!
	&= r! \prod_{i = r + 1}^{p - 1} \grayparen{p - i} \\
	&\equiv r! \graysizedparen{\prod_{i = r + 1}^{p - 1} \grayparen{-i} + p \, \grayparen{-1}^{p - 1 - r} \frac{\grayparen{p - 1}!}{r!} \sum_{i = r + 1}^{p - 1} \frac{1}{-i}} \mod p^2 \\
	&= \grayparen{-1}^{p - 1 - r} \grayparen{p - 1}! \graysizedparen{1 - p \, \grayparen{H_{p - 1} - H_r}}.
\end{align*}
Therefore
\begin{align*}
	\frac{r! \grayparen{p - 1 - r}!}{s! \grayparen{p - 1 - s}!}
	&\equiv \grayparen{-1}^{r - s} \frac{1 - p \, \grayparen{H_{p - 1} - H_r}}{1 - p \, \grayparen{H_{p - 1} - H_s}} \mod p^2 \\
	&\equiv \grayparen{-1}^{r - s} \grayparen{1 - p \, \grayparen{H_{p - 1} - H_r}} \grayparen{1 + p \, \grayparen{H_{p - 1} - H_s}} \mod p^2 \\
	&\equiv \grayparen{-1}^{r - s} \grayparen{1 + p H_r - p H_s} \mod p^2.
\end{align*}
This is equivalent to
\[
	\frac{r!}{s!} \equiv \grayparen{-1}^{r - s} \frac{\grayparen{p - 1 - s}!}{\grayparen{p - 1 - r}!} \grayparen{1 + p H_r - p H_s} \mod p^2.
\]
Dividing both sides by $\grayparen{r - s}!$ produces Congruence~\eqref{rotation congruence modulo p^2}.
\end{proof}

Modulo~$p$, we obtain the following rotational symmetry.

\begin{corollary}\label{rotation}
Let $p$ be a prime.
If $0 \leq s \leq r \leq p - 1$, then
\[
	\binom{r}{s}
	\equiv \grayparen{-1}^{r - s} \binom{p - 1 - s}{r - s}
	\mod p.
\]
\end{corollary}

We now use Theorem~\ref{general congruence} and Proposition~\ref{general rotation} to prove Theorem~\ref{characterization}, adding a third equivalent statement.
We assume $s \leq r$, since otherwise $(r, s) \notin D(p)$ by Proposition~\ref{zero region}.

\begin{theorem}\label{characterization with third statement}
Let $p$ be a prime, and let $0 \leq s \leq r \leq p - 1$.
The following are equivalent.
\begin{itemize}
\item[1.]
$(r, s) \in D(p)$.
\item[2.]
$H_r \equiv H_{r - s} \equiv H_s \mod p$.
\item[3.]
$
	\binom{r}{s}
	\equiv \grayparen{-1}^{r - s} \binom{p - 1 - s}{r - s}
	\equiv \grayparen{-1}^s \binom{p - 1 - r + s}{s} \mod p^2
$.
\end{itemize}
\end{theorem}

\begin{proof}
First we show that
\begin{equation}\label{Lucas congruence modulo p^2}
	\binom{p a + r}{p b + s} \equiv \binom{a}{b} \binom{r}{s} \mod p^2
\end{equation}
for all $a \geq 0$ and $b \geq 0$ if and only if $H_r \equiv H_{r - s} \equiv H_s \mod p$.
By Theorem~\ref{general congruence},
\[
	\binom{p a + r}{p b + s}
	\equiv \binom{a}{b} \binom{r}{s} \graysizedparen{1 + p a \, \grayparen{H_r - H_{r - s}} + p b \, \grayparen{H_{r - s} - H_s}} \mod p^2.
\]
Clearly, if $H_r \equiv H_{r - s} \equiv H_s \mod p$, then Congruence~\eqref{Lucas congruence modulo p^2} holds for all $a \geq 0$ and $b \geq 0$.
Conversely, assume Congruence~\eqref{Lucas congruence modulo p^2} holds for all $a \geq 0$ and $b \geq 0$.
Since $\binom{r}{s}$ is not divisible by $p$, this, along with Theorem~\ref{general congruence}, implies
\[
	\binom{a}{b}
	\equiv \binom{a}{b} \graysizedparen{1 + p a \, \grayparen{H_r - H_{r - s}} + p b \, \grayparen{H_{r - s} - H_s}} \mod p^2.
\]
Setting $a = 1$ and $b = 0$ shows that $H_r \equiv H_{r - s} \mod p$.
Now setting $a = 1$ and $b = 1$ shows that $H_{r - s} \equiv H_s \mod p$.

Next we show the equivalence of the second and third statements.
We see from Proposition~\ref{general rotation} that $H_r \equiv H_s \mod p$ if and only if
\[
	\binom{r}{s}
	\equiv \grayparen{-1}^{r - s} \binom{p - 1 - s}{r - s} \mod p^2.
\]
Similarly, $H_r \equiv H_{r - s} \mod p$ if and only if
\[
	\binom{r}{r - s}
	\equiv \grayparen{-1}^s \binom{p - 1 - r + s}{s} \mod p^2.
\]
Since $\binom{r}{r - s} = \binom{r}{s}$, this implies that $H_r \equiv H_{r - s} \equiv H_s \mod p$ if and only if
$
	\binom{r}{s}
	\equiv \grayparen{-1}^{r - s} \binom{p - 1 - s}{r - s}
	\equiv \grayparen{-1}^s \binom{p - 1 - r + s}{s} \mod p^2
$.
\end{proof}

Theorem~\ref{characterization with third statement} and Proposition~\ref{reflection symmetry} imply that $D(p)$ is invariant under the symmetries of the equilateral triangle.

We conclude by returning to the discussion of Wieferich primes.
Eisenstein~\cite{Eisenstein} showed that $H_{\grayparen{p - 1}/2} \equiv \frac{2 - 2^p}{p} \mod p$ for $p \geq 3$.
Therefore $p$ is a Wieferich prime if and only if $H_{\grayparen{p - 1}/2} \equiv 0 \mod p$, which is equivalent to $(\frac{p - 1}{2}, \frac{p - 1}{2}) \in D(p)$ by Theorem~\ref{characterization with third statement}.
By rotational symmetry, $p$ is a Wieferich prime if and only if
\[
	\left\{(\tfrac{p - 1}{2}, 0), (\tfrac{p - 1}{2}, \tfrac{p - 1}{2}), (p - 1, \tfrac{p - 1}{2})\right\} \subseteq D(p).
\]

\section{Acknowledgment}

Thanks to Erin Craig for excellent input on the design of the graphics and for improvements to the presentation.

\end{document}